\documentclass[12pt,  reqno]{amsart}
\usepackage[utf8]{inputenc}
\usepackage{amsthm,amsfonts,amssymb}
\usepackage{enumerate}
\usepackage{mathrsfs}
\usepackage[bookmarksnumbered, colorlinks, plainpages]{hyperref}
\hypersetup{colorlinks=true,linkcolor=red, anchorcolor=green, citecolor=cyan, urlcolor=red, filecolor=magenta, pdftoolbar=true}

\newtheorem{theorem}{Theorem}[section]
\newtheorem{lemma}[theorem]{Lemma}
\newtheorem{proposition}[theorem]{Proposition}
\newtheorem{corollary}[theorem]{Corollary}
\newtheorem{definition}[theorem]{Definition}
\newtheorem{example}[theorem]{Example}

\newtheorem{remark}[theorem]{Remark}

\newcommand{\cpcc}[3]{\mathcal{C}\mathcal{P}\mathcal{C}\mathcal{C}_{\text{loc}}({#1},~ \mathcal{C}_{\mathcal{#2}}^* (\mathcal{#3}))}

\newcommand{\ced}[2]{\mathcal{C}_{\mathcal{#1}}^* (\mathcal{#2})}

\author{Arunkumar  C.S.}

\address{Kerala School of Mathematics, Kozhikode 673 571, India.}
\email{\textcolor[rgb]{0.00,0.00,0.84}{arunkumarcsmaths9@gmail.com}} 



\begin{document}

\setcounter{page}{1}

\title{Local boundary representations  of locally C*-algebras}

\maketitle

\begin{abstract}We initiate a study of non-commutative Choquet boundary for spaces of unbounded operators. We define the notion of local boundary representations for local operator systems in locally C$^*$-algebras and prove that  local boundary representations provide an intrinsic invariant for a particular class of local operator systems. An appropriate analog of purity of  local completely positive maps on local operator systems is used to characterize local boundary representations for local operator systems in Frechet locally C$^*$-algebras. 
\end{abstract}

\section{Introduction}
The notion of locally $C^*$-algebras was introduced by Atushi Inoue \cite{Inou71} to study algebras of unbounded operators on  a Hilbert space. In the literature,  locally $C^*$-algebra have been studied by several authors under different names like pro-$C^*$-algebras, $O^*$-algebras, $LCM^*$-algebras, and multinormed $C^*$-algebras. Effros and Webster \cite{Effr96}  initiated a study of the locally convex version of  operator spaces  called  the \textit{local operator spaces}. In 2008, A. Dosiev \cite{Dosi08} realized local operator spaces as  subspaces of the locally $C^*$-algebra  $\ced{E}{D}$ of unbounded operators on a quantized domain $\mathcal{E}$ with its union space $\mathcal{D}$.  Also, Dosiev introduced \textit{local operator systems} as the  unital self adjoint subspaces of $\ced{E}{D}$. Based upon the local positivity concept in locally $C^*$-algebra,  Dosiev \cite{Dosi08} proved Stinespring representation theorem for local completely positive maps and  Arveson extension theorem for local completely positive maps on Frechet local operator systems. Recently, the  \text{minimality} of Stinespring representation was identified   by Bhat, Anindya, and Santhoshkumar \cite{Bhat21}, and  showed that minimal Stinespring's representation is unique up to unitary equivalence.

The extremal theory concerning Choquet boundary of subalgebras of function algebras play an important role in numerous areas of classical analysis. Let $X$ be a locally compact Hausdorff space and $C(X)$ be the algebra of all continuous functions on $X$. Given a uniform algebra $\mathcal{U} \subseteq C(X)$ and a point $x_{0}\in X$. If the evaluation functional corresponds to $x_{0}$ admits a unique completely positive extension from $\mathcal{U}$ to $C(X)$, then  we say that the  point $x_{0}$ is in the \textit{Choquet boundary}\cite{BL75} of $\mathcal{U}$. The non-commutative analog of this notion called the \textit{boundary representations} of linear subspaces  in $C^*$-algebras was introduced  by Arveson \cite{Arv69} and studied extensively by him in  \cite{Arv72,Arv08}. The objects boundary representations are intrinsic invariants for operator systems(and operator algebras), and  provide a context for showing the existence of non-commutative Silov boundary. The articles \cite{DK15,DM05,MH79,MS98} are also worth mentioning in this context. There is  plenty of literature on generalizing the  notion  of boundary representations to different contexts
\cite{Arun21,FHL18,NPSV18}. This article  initiate a study of   non-commutative Choquet boundary in the context of locally $C^*$-algebras and a related extremal notion of purity of local completely positive maps. 

This paper is organized as follows. In  section 2, we recall necessary background material and results that are required throughout. Section 3 deals with certain elementary results on  local completely contractive(local CC) maps and local completely positive(local CP) maps. We obtain  a locally convex version of the  Arveson extension theorem for  unital local CC-maps. That is, a unital local CC-map  from a subspace  $\mathcal{M}$ of  a locally $C^*$-algebra
$\mathcal{A}$ can be extended to a local CP-map on $\mathcal{A}$. In Section 4, a suitable notion of  \textit{irreducible representations} of locally $C^*$-algebras is introduced  using the idea of commutants. We show that an irreducible Stinespring representation of a Frechet locally $C^*$-algebra is minimal.  The concept of \textit{pure local CP-maps} on local operator systems are introduced and proved  that a local CP-map on a locally $C^*$-algebra is a pure local CP-map if and only if its  minimal Stinespring representation is irreducible. In Section 5, we introduce \textit{local boundary representations} of locally $C^*$-algebras and prove that local boundary representations provide an intrinsic invariant for local operator systems. In the case of Frechet locally $C^*$-algebras, we characterize local boundary representations using the notions of  pure local CP-maps and a couple of other new notions. 

\section{Preliminaries}
\subsection{Locally C*-algebras}Let $\mathcal{A}$ be a unital $*$-algebra with unit $1_{\mathcal{A}}$. A seminorm $p$ on $\mathcal{A}$ is said to be sub-multiplicative, if  $p(1_{\mathcal{A}})=1$ and $p(ab)\leq p(a)p(b)$ for every $a,b\in \mathcal{A}$. A sub-multiplicative seminorm   $p$  satisfies the condition  $p(a^*a)=p(a)^2$ for every $a\in \mathcal{A}$, is called a $C^*$-seminorm. Let $(\Lambda,\leq)$ be a directed poset.  A family of seminorms $\mathcal{P}=\{p_{\alpha}:\alpha\in \Lambda\}$ on $\mathcal{A}$ is called an upward filtered family, if $\alpha \leq \beta$ in $\Lambda$, then 
$p_{\alpha}(a)\leq p_{\beta}(a)$ for every $a \in \mathcal{A}$. A \textit{locally  C$^*$-algebra} $\mathcal{A}$ is a $*$-algebra together with an upward filtered  family of $C^*$-seminorms $\mathcal{P}$ on $\mathcal{A}$ such that $\mathcal{A}$ is complete with respect to the locally convex topology generated by the family $\mathcal{P}$. 

Throughout this article,  $\mathcal{A}$ always denote a  locally $C^*$-algebra with a prescribed family of $C^*$seminorms $\{p_{\alpha}:\alpha\in \Lambda\}$. Let $I_{\alpha}=\{a\in \mathcal{A}:p_{\alpha}(a)=0\}$ and $\mathcal{A}_{\alpha}$  be the quotient  $C^*$-algebra $\mathcal{A}/I_{\alpha}$ with the $C^*$-norm induced by $p_{\alpha}$. Denote the cannonical quotient $*$-homomorphism  from $\mathcal{A}$ to $\mathcal{A}_{\alpha} $ by $\pi_{\alpha}$. Note that for $\alpha \leq \beta$ in $\Lambda$, there is a   cannonical $*$-homomorphism  $\pi_{\alpha \beta}:\mathcal{A}_{\beta} \rightarrow \mathcal{A}_{\alpha}$ where $\pi_{\alpha \beta}(a+I_{\beta})=a+I_{\alpha}$ and that  satisfies  $\pi_{\alpha \beta}\pi_{\beta}=\pi_{\alpha}$. Then
one can identify $\mathcal{A}$ as the  the inverse limit of the projective system $\{\mathcal{A}_{\alpha},\pi_{\alpha, \beta}: \alpha, \beta\in \Lambda \}$ of $C^*$-algebras \cite{Phil88}.  

\subsection{Local positve elements}Anar Dosiev \cite{Dosi08} introduced the notions of local  hermitian and local positivity in locally $C^*$-algebras. An element $a\in \mathcal{A}$ is called \textit{local hermitian} if $a=a^*+x$ for some $x\in \mathcal{A}$ such that $p_{\alpha}(x)=0$ for some $\alpha\in \Lambda$ and an element $a\in \mathcal{A}$ is called \textit{local positive} if $a=b^*b+x$ for some $b,x\in \mathcal{A}$ such that $p_{\alpha}(x)=0$ for some $\alpha\in \Lambda$.  In this case, we call $a$ is $\alpha$-hermition (and $\alpha$-positive, respectively). We use  $a\geq_{\alpha}0$ to denote $a$ is $\alpha$-positive. A direct computation shows that $a\geq_{\alpha}0$ in $ \mathcal{A}$ if and only if the $\pi_{\alpha}(a) \geq 0$  in the $C^*$-algebra $\mathcal{A}_{\alpha}$. 
\subsection{Local operator systems and local CP-maps}Let $\mathcal{A}$ be a locally $C^*$-algebra. For  a  linear subspace $S$ of $\mathcal{A}$ denote $S^*=\{x^*:x\in S\}$. We say $S$ is \textit{self adjoint} if $S=S^*$. A \textit{local operator system} in $\mathcal{A}$ is a unital self adjoint linear subspace  of $\mathcal{A}$. An element $a$ in a local operator system $S$ is local positive if $a$ is local positive in $\mathcal{A}$. 
 Consider another locally $C^*$-algebra $\mathcal{B}$  with the associated  family of seminorms  $\{q_{l}:l\in \Omega\}$. Let $S_{1}$ and $S_{2}$ be local operator systems in $\mathcal{A}$ and $\mathcal{B}$ respectively.
 A linear map $\phi:S_{1}\rightarrow S_{2}$ is  said to be  \textit{local positive}, if for each $l\in \Omega$ there corresponds $\alpha\in \Lambda$ such that $\phi(a)\geq_{l}0$ whenever $a\geq_{\alpha} 0$ in $S_{1}$. The map  $\phi$ is said to be \textit{local bounded}, if for each $l\in \Omega$ there exists an $\alpha\in\Lambda$ and $C_{l\alpha}>0$ such that $q_{l}(\phi(a))\leq C_{l\alpha}p_{\alpha}(a)$ for all $a\in S_{1}$. If $C_{l\alpha}$ can be chosen to be 1, then we say that $\phi$ is \textit{local contractive}. For $n\in \mathbb{N}$, let $M_{n}(\mathcal{A})$ denotes the set of all $n\times n$ matrices over $\mathcal{A}$. Naturally $M_{n}(\mathcal{A})$  is  a locally $C^*$-algebra with the    defining family of seminorms $\{p_{\alpha}^n:\alpha\in \Lambda\}$, where  $p_{\alpha}^n([a_{ij}])=\Vert \pi_{\alpha}^{(n)}([a_{ij}]) \Vert_{\alpha} $ for $[a_{ij}]$ in  $M_{n}(\mathcal{A})$. We use $\phi^{(n)}$ to denote the  $n$-amplification of the map $\phi$, that is, $\phi^{(n)}: M_{n}(S_{1})\rightarrow M_{n}(S_{2})$ 
defined by $\phi^{(n)}([a_{ij}])=[\phi(a_{ij})]$ for 
$[a_{ij}]$ in  $M_{n}(S_{1})$. The map $\phi$ is called  \textit{local completely bounded}(local CB-map) if for each $l\in \Omega$, there exists $\alpha\in \Lambda$ and $C_{l\alpha}>0$ such that $q_{l}^{n}([\phi(a_{ij})])\leq C_{l\alpha}p_{l}^{n}([a_{ij}])$, for every $n\in \mathbb{N}$. If $C_{l\alpha}$ can be chosen to be $1$, then we say $\phi$ is  \textit{local completely contractive}(local CC-map). The map $\phi$ is called  \textit{local completely positive}(local CP-map) if for each $l\in \Omega$, there exists $\alpha\in \Lambda$ such that $\phi^{(n)}([a_{ij}])\geq_{l}0$ in $M_{n}(S_{2})$ whenever $[a_{ij}]\geq_{\alpha}0$ in $M_{n}(S_{1})$.

\subsection{Representations of locally C*-algebras} Let $H$ be a complex Hilbert space and $\mathcal{D}$ be a dense subspace of $H$. A \textit{quantized domain} in $H$ is  a triple $\{H,\mathcal{E}, \mathcal{D}\}$, where $\mathcal{E}=\{H_{l}: l\in \Omega\}$ is an upward filtered family of closed subspaces of $H$ such that the union space $\mathcal{D}=\bigcup\limits_{l\in \Omega}H_{l}$ is dense in $H$.  In short, we say  $\mathcal{E}$ is a quantized domain in $H$ with its union space $D$.  
A quantized doamin $\mathcal{E}$ is called a \textit{quantized Frechet domain} if $\mathcal{E}$ is a countable family. 

Corresponding to a quantized domain $\mathcal{E}=\{H_{l}: l\in \Omega\}$ we can associate an upward filtered family  $\mathscr{P}=\{P_{l}: l\in \Omega\}$ of projections in $B(H)$ where $P_{l}$ is the orthogonal projection of $H$ onto the closed subspace $H_{l}$.
\subsection*{The space $\ced{E}{D}$} Let us  denote $L(\mathcal{D})$ by  the set of all linear operators on the linear subspace $\mathcal{D}$. The set  of all \textit{noncommutative continuous functions} on a quantized domain $\mathcal{E}$ is defined as 
$$\mathcal{C}_{\mathcal{D}}(\mathcal{E})=\{T\in L(\mathcal{D}): TP_{l}=P_{l} TP_{l}\in B(H), \text{ for all } l\in \Omega\}.$$ 
Note that  $\mathcal{C}_{\mathcal{D}}(\mathcal{E})$ is an algebra and if $T\in L(\mathcal{D})$, then
$$T\in \mathcal{C}_{\mathcal{D}}(\mathcal{E}) \text{ if and only if } T(H_{l})\subseteq H_{l} \text{ and } T|_{H_{l}}\in B(H_{l}) \text{ for all } l\in \Omega.$$
The  \textit{$*$-algebra of all noncommutative continuous functions on a quantized domain} $\mathcal{E}$  is defined as 
$$\ced{E}{D}=\{T\in \mathcal{C}_{\mathcal{D}}(\mathcal{E}) : P_{l}T\subseteq TP_{l}, \text{ for all } l\in \Omega\}.$$ 
Note that $\ced{E}{D}$ is  a unital subalgebra of  $\mathcal{C}_{\mathcal{D}}(\mathcal{E})$. For more details about the adjoint of operators in $\ced{E}{D}$  refer \cite[Proposition 3.1]{Dosi08}. For $T\in L(\mathcal{D})$, it is easy to see that 
$T\in \ced{E}{D}$  if and only if  for all  $l\in \Omega$
$$T(H_{l})\subseteq H_{l},~  T|_{H_{l}}\in B(H_{l}) \text{ and } T(H_{l}^{\perp}\cap\mathcal{D})\subseteq H_{l}^{\perp}\cap\mathcal{D}.$$
Now, define $q_{l}: \ced{E}{D}\rightarrow \mathbb{R}$ by $q_{l}(T)=\Vert T|_{H_{l}} \Vert$ for all $T\in \ced{E}{D}$. Then $\mathcal{Q}=\{q_{l}:l\in \Omega \}$ is an upward filtered family of $C^*$-seminorms on $\ced{E}{D}$.  Also, $\ced{E}{D}$ is complete with respect to the locally convex topology generated by the family  $\mathcal{Q}$. Hence $\ced{E}{D}$ is a  locally $C^*$-algebra.   

We use $\cpcc{S}{E}{D}$ to denotes the class of all local completely positive  and local completely contractive maps from a local operator system $S$ to $\mathcal{C}_{\mathcal{E}}^* (\mathcal{D})$. 

\subsection*{Stinespring's theorem for local CP-maps} A locally convex version (or an unbounded version)  of the celebrated  Stinespring's dilation theorem is appeared in the work of A.Dosiev \cite[Theorem 5.1]{Dosi08}. 
\begin{theorem}\cite[Theorem 5.1]{Dosi08}\label{Stinespring's theorem}
Let $\phi\in\cpcc{\mathcal{A}}{E}{D}$. Then there exists a Hilbert space $H^{\phi}$ and a quantized domain $\mathcal{E}^{\phi}=\{H_{\alpha}^{\phi}:\alpha\in \Lambda\}$ in $H^{\phi}$   with its union space $\mathcal{D}^{\phi}$,  a contraction $V_{\phi}:H\rightarrow H^{\phi}$, and a unital local contractive $*$-homomorphism $\pi_{\phi}:\mathcal{A}\rightarrow C^*_{\mathcal{E}^{\phi}}(\mathcal{D}^{\phi})$ such that 
$$ \phi(a)\subseteq V_{\phi}^*\pi_{\phi}(a)V_{\phi} \text{ and } V_{\phi}(H_{\alpha})\subseteq H_{\alpha}^{\phi}$$
for every $a\in\mathcal{A}$ and $l\in \Lambda$.
Moreover, if $\phi(1_{\mathcal{A}})=1_{\mathcal{D}}$, then $V_{\phi}$ is an isometry.
\end{theorem}


 Any triple $(\pi_{\phi},V_{\phi},\{H^{\phi};\mathcal{E}^{\phi};\mathcal{D}^{\phi}\})$  that satisfies the conditions of the Theorem \ref{Stinespring's theorem} is called  a \textit{Stinespring representation} for $\phi$. 
\subsection*{Minimality of Stinespring representation for local CP-maps} 
The \textit{minimality} of the Stinespring representation was introduced and studied recently by Bhat and et al in \cite{Bhat21}. 
A Stinespring representation $(\pi_{\phi},V_{\phi},\{H^{\phi}; \mathcal{E}^{\phi} ;\mathcal{D}^{\phi}\})$ of $\phi$ is said to be \textit{minimal}, if $H_{l}^{\phi}=[\pi_{\phi}V_{\phi}H_{l}]$, for every $l\in \Lambda$. They proved that given any Stinespring represesentation of a map $\phi\in\cpcc{\mathcal{A}}{E}{D}$, one can reduce it to  minimal Stinespring representation, and also any two minimal Stinespring representations are unitarily equivalent in the following sense. 
Let $\pi_{1}$ and $\pi_{2}$ be two representations of  the locally $C^*$-algebra $\mathcal{A}$ on the quanitzed doamins $\{H;\mathcal{E}=\{H_{l}:l\in \Omega\};\mathcal{D}\}$ and $\{H';\mathcal{E}'=\{H'_{l}:l\in \Omega \};\mathcal{D}'\}$, respectively. We say $\pi_{1}$ and $\pi_{1}$ are unitarily equivalent if there exists a unitary $U:H'\rightarrow H$  such that $U(H'_{l})\subseteq H_{l}$  and  $\pi_{2}(a)=U^*\pi_{1}(a)U|_{\mathcal{D}'}$ for all $a\in\mathcal{A}$ and  all $l\in \Omega$.

\section{ Local positive linear maps}
In this section, we prove an analog of the Arveson extension theorem for local CC-maps on linear subspaces of $\ced{E}{D}$ for a quantized Frechet domain $\mathcal{E}$. This result is crucial in establishing a theorem in the main section. Now, let $S$ be a local operator system in the locally $C^*$-algebra $\mathcal{A}$. A linear functional $f:S\rightarrow\mathbb{C}$ is an $\alpha$-contractive linear functional if $\vert f(a)\vert\leq p_{\alpha}(a)$ for all $a\in S$. Note that, by Hahn-Banach extension theorem, there is an $\alpha$-contractive linear map $\tilde{f}:\mathcal{A}\rightarrow\mathbb{C}$  such that $\tilde{f}|_{S}=f$ and $\vert \tilde{f}(a)\vert\leq p_{\alpha}(a)$ for all $a\in \mathcal{A}$.  For $a\in\mathcal{A}$ we define  the $\alpha$-\textit{spectrum} of $a$  to be the spectrum of $\pi_{\alpha}(a)$ in the $C^*$-algebra $\mathcal{A}_{\alpha}$. We use $\sigma_{\alpha}(a)$ to denote the $\alpha$-spectrum of $a$. 

\begin{lemma}\label{lc implies lp for linear functional}
Let $S$ be a local operator system in a locally $C^*$-algebra $\mathcal{A}$ and let $f:S\rightarrow \mathbb{C}$ be a unital $\alpha$-contractive linear functional. Let $\tilde{f}$ be a Hahn-Banach extension of $f$ to $\mathcal{A}$. If $a=x^*x+b\in S$ is an  $\alpha$-positive element of $\mathcal{A}$, then $0\leq \tilde{f}(x^*x)\leq r_{\alpha}$,  where $r_{\alpha}$ is the spectral radius of  $\pi_{\alpha}(a)$.
\end{lemma}
\begin{proof} Assume that  $\tilde{f}(x^*x)\notin [0, r_{\alpha}]$. Since a closed interval in the real line is the intersection of all closed disks containing it in the complex plane, there exists a closed disk $D_{r}(\mu)$ centered at $\mu\in \mathbb{C}$ and radius $r$ such that $|\tilde{f}(x^*x)-\mu|>r$ and $[0,r_{\alpha}]\subseteq D_{r}(\mu)$. Then $\sigma_{\alpha}(x^*x-\mu 1)\subseteq D_{r}(0)$ as $\sigma_{\alpha}(x^*x)\subseteq [0, r_{\alpha}] \subseteq D_{r}(\mu)$. Since $\pi_{\alpha}(x^*x)$ is a positive element of $\mathcal{A}_{\alpha}$, $\pi_{\alpha}(x^*x-\mu 1)$  is a normal element of $\mathcal{A}_{\alpha}$. The spectral radius and norm are same for normal elements of a $C^*$-algebra gives us $\Vert \pi_{\alpha}( x^*x-\mu 1  )\Vert _{\alpha} \leq r$. Now using the fact $\tilde{f}$ is a unital $\alpha$-contraction, we have
\begin{align*}
 |\tilde{f}(x^*x)-\mu|&=|\tilde{f}(x^*x-\mu 1)|\leq  p_{\alpha}(x^*x-\mu 1) \\
  &= \Vert \pi_{\alpha}(x^*x-\mu 1)\Vert_{\alpha} \leq r.
\end{align*}
This is a contradiction. Hence $\tilde{f}(x^*x)\in [0,r_{\alpha}]$.
\end{proof}

\begin{theorem}\label{lc implies lp}
Let $S$ be a local operator system in a locally $C^*$-algebra $\mathcal{A}$ and $\mathcal{E}$ be a quantized domain with its union space $\mathcal{D}$. Let $\phi: S\rightarrow \ced{E}{D}$ be a unital  local contractive map. Then  $\phi$ is a local positive map.
\end{theorem}
\begin{proof} Fix $l\in \Omega$. Since $\phi$ is local contractive, there exists $\alpha\in \Lambda$ such that $\Vert \phi(a)\Vert_{l} \leq p_{\alpha}(a)$ for every $a\in S$. Let  $a\in S$ and $a=x^*x+b$ where $x,b\in \mathcal{A}$ and $p_{\alpha}(b)=0$ for some $\alpha\in \Lambda$. 
 
 We will show that $\phi(a)|_{H_{l}}$ is a positive operator on  $H_{l}$. 
 Let $h\in H_{l}$ with $\Vert h \Vert =1$.  Define $f_{h}:S\rightarrow \mathbb{C}$ by $f_{h}(y)=\langle \phi(y)|_{H_{l}}h,h \rangle$. Then $f_{h}(1)=1$ and $$|f_{h}(y)|\leq \Vert \phi(y)\Vert_{l}\leq p_{\alpha}(y).$$
 Therefore, the linear functional $f_{h}$ is a unital  $\alpha$-contraction.  Let $\tilde{f}_{h}:\mathcal{A}\rightarrow \mathbb{C}$  be an $\alpha$-contractive Hahn-Banach extension of $f_{h}$. Then 
 $$ \langle \phi(a)|_{H_{l}}h,h\rangle =f_{h}(a) =\tilde{f}_{h}(a) =\tilde{f}_{h}(x^*x)+\tilde{f}_{h}(b).
 $$
Note that, $\tilde{f}_{h}(b)=0$ as $\tilde{f}_{h}$ is an $\alpha$-contraction and  $p_{\alpha}(b)=0$. Using 
Lemma \ref{lc implies lp for linear functional} we conclude that 
$\tilde{f}_{h}(x^*x)=\langle \phi(a)|_{H_{l}}h,h\rangle $ is positive. Therefore, $\phi(a)$ is local positive and that completes the proof. 
\end{proof}  

\begin{remark}
We can use the above theorem to establish the following result, which is a  special case of a result in \cite{Dosi08}.
\end{remark}
 
\begin{theorem}\label{local CC iff local CP}\cite[Corollary 4.1]{Dosi08}
Let $S$ be a local operator system in a locally $C^*$-algebra $\mathcal{A}$ and $\mathcal{E}$ be a quantized domain with its union space $\mathcal{D}$. Let $\phi: S\rightarrow \ced{E}{D}$ be a unital linear map. Then $\phi$ is a local CC-map if and only if   $\phi$ is  a local CP-map.
\end{theorem}
\begin{proof}
Let $\phi$ be a local CC-map. Fix $l\in \Omega$.  There exists a $\alpha \in \Lambda$  such that $\Vert \phi^{(n)}([a_{ij}])\Vert_{l}\leq p^{(n)}_{\alpha}([a_{ij}])$ for all $[a_{ij}]\in M_{n}(S)$,$n\in\mathbb{N}$.  From the proof of  Theorem \ref{lc implies lp} we have  $\phi^{(n)}([a_{ij}])\geq_{l}0$ whenver $[a_{ij}]\geq_{\alpha}0$.   Thus $\phi$ is a local CP-map.

Conversely, assume that $\phi$ is a local CP-map. 
Fix $l\in \Omega$.  There exists a $\alpha \in \Lambda$  such that $ \phi^{(n)}(A)\geq_{l}0$ whenever $A\geq_{\alpha}0$ in $A\in M_{n}(S)$ and $n\in\mathbb{N}$. 
Let  $A\in M_{n}(S)$ such that $p^{(n)}_{\alpha}(A)\leq 1$. Then 
$$\begin{bmatrix} 1_{n} & A \\
A^* & 1_{n}
\end{bmatrix} \geq_{\alpha} 0\text{ in } M_{2n}(S) $$
Applying the map $\phi^{(2n)}$, we   have $$\begin{bmatrix} I_{n} & \phi^{(n)}(A) \\
\phi^{(n)}(A^*) & I_{n}
\end{bmatrix} \geq_{l} 0 \text{ in }M_{2n}(\ced{E}{D}).$$
Thus $$\begin{bmatrix} I_{n} & \phi^{(n)}(A) \\
\phi^{(n)}(A^*) & I_{n}
\end{bmatrix}_{\bigg|_{H^{n}_{l}\oplus H^{n}_{l}}} \geq  0 \text{ in } B(H^{n}_{l}\oplus H^{n}_{l}).$$
Equivalently $\Vert \phi^{n}(A)|_{H^{n}_{l}}\Vert\leq 1$.  Hence  $\Vert \phi^{n}(A)\Vert_{l}\leq p^{(n)}_{\alpha}(A)$ for every $A\in M_{n}(S)$.
That is, $\phi$ is a local CC-map.
\end{proof}

\begin{theorem}\label{extension to operator system}
Let $\mathcal{A}$   be  unital a locally $C^*$-algebra, and $M$ be a unital subspace of  $\mathcal{A}$. If $\phi : M\rightarrow \ced{E}{D}$ be a unital local contraction, then there is a  local positive extension $\tilde{\phi}$ of $\phi$ to $M+M^*$ given by  
$\tilde{\phi}(x+y^*)=\phi(x)+\phi(y)^*.$
Moreover, $\tilde{\phi}$ is the only  local positive extension of $\phi$ to $M+M^*$. 
\end{theorem}
\begin{proof}
First, we will show that the map $\tilde{\phi}$ is  well-defined. Let $$M_{*}=\{a\in M: a^*\in M\}.$$ Clearly, $M_{*}$ is a local operator system in  $\mathcal{A}$. Also, the map $\phi$ is  a unital local contractive map on $M_{*}$. Using Theorem \ref{lc implies lp} we have $\phi$ is a local positive map. Then $\phi$ is self adjoint on $M_{*}$, thanks to 
\cite[Lemma 4.3]{Dosi08}. To see $\tilde{\phi}$ is well defined,  consider $a_{1},a_{2},b_{1},b_{2}\in M$ with $a_{1}+b_{1}^*=a_{2}+b_{2}^*$. Equivalently, $a_{1}-a_{2}=(b_{2}-b_{1})^*$. Thus $b_{2}-b_{1} \in M_{*}$. Then using the fact that $\phi$ is self adjoint on $M_{*}$, we have
\begin{align*}
\phi(a_{1}-a_{2})&=\phi((b_{2}-b_{1})^*)\\
                &= [\phi(b_{2}-b_{1})]^* \\
                 &= \phi(b_{2})^*-\phi(b_{1})^*\\
\phi(a_{1})+\phi(b_{1})^*&=\phi(a_{2})+\phi(b_{2})^*.
\end{align*}
Hence $\tilde{\phi}(a_{1}+b_{1}^*)=\tilde{\phi}(a_{2}+b_{2}^*)$. That is, $\tilde{\phi}$ is well-defined.

To see $\tilde{\phi}$ is local positive; fix $l\in \Omega$. By local contractivity of $\phi$, there exists an $\alpha\in \Lambda$ such that $\Vert \phi(a)\Vert_{l}\leq p_{\alpha}(a)$ for all $a\in \mathcal{A}$. Let $a+b^*\in M+M^*$ be an $\alpha$-positive element. We will show that $\tilde{\phi}(a+b^*)$ is   local positive by showing that $\tilde{\phi}(a+b^*)|_{H_{l}}$ is a positive operator on $H_{l}$. Let $h\in H_{l}$ with $\Vert h \Vert=1$. Define $f:M\rightarrow \mathbb{C}$ by $f(y)=\langle \phi(y)h ,h\rangle $. Then $\vert f(y) \vert \leq \Vert \phi(y) \Vert_{l}\leq p_{\alpha}(y)$ for every $y\in M$. Using Hahn-Banach extension theorem, $f$ extends to $f_{1}:M+M^*\rightarrow \mathbb{C}$ with 
$\vert f_{1}(y) \vert  \leq p_{\alpha}(y)$ for every $y\in M+M^*$. By Theorem \ref{lc implies lp} we have that $f_{1}$ is local positive. Also, $0\leq f_{1}(a+b^*)=f_{1}(a)+\overline{f_{1}(b)}=f(a)+\overline{f(b)}=\langle \phi(a)h ,h\rangle +\overline{\langle \phi(b)h ,h\rangle }=\langle \tilde{\phi}(a+b^*)h ,h\rangle$. Hence $\tilde{\phi}$ is local positive.

To show $\tilde{\phi}$ is unique; let $\psi:M+M^*\rightarrow \mathbb{C}$ be  a local positive  extension of $\phi$.  The map $\psi$ is self adjoint  by \cite[Lemma 4.3]{Dosi08}. Then the following  computation shows  that  $\psi=\tilde{\phi}$. 
\begin{align*}
 \psi(a+b^*)&=\psi(a)+\psi(b^*) =\psi(a)+\psi(b)^*\\
        &=\phi(a)+\phi(b)^*= \tilde{\phi}(a+b^*).
\end{align*}
\end{proof}

Let   $\mathcal{F}$ be a quantized Frechet domain with its union space $\mathcal{O}$.   A. Dosiev  \cite[Theorem 8.2]{Dosi08} proved the analog of  Arveson's extension theorem  for unital local CP-maps  from  local operator systems into  $\ced{F}{O}$.  Using the above theorem we deduce an analog  of Arvesion extension theorem for local CC-maps on subspaces of locally $C^*$-algebras.  A locally $C^*$-algebra $\mathcal{A}$ is called Frechet locally $C^*$-algebra  if there is a local isometrical $*$-homomorphism $\mathcal{A}\rightarrow \ced{E}{D}$ for some quantized  Frechet domain $\mathcal{E}$ with its union space $\mathcal{D}$.

\begin{theorem}\label{local CP-extension of local CC-map}
Let   $\mathcal{F}$ be a quantized Frechet domain and $\mathcal{A}$ be a  Frechet locally $C^*$-algebra. Let $M$ be a unital linear subspace of  $\mathcal{A}$ and  $\phi:M\rightarrow \ced{F}{O}$ be a unital local CC-map. Then $\phi$ has a local CP-extension to $\mathcal{A}$. 
\end{theorem}
\begin{proof}
Since $\phi$ is local CC-map, by Theorem \ref{extension to operator system} there is a local CP-map $\tilde{\phi}:M+M^*\rightarrow \ced{F}{O}$. Then by Dosiev-Arveson extension theorem  \cite[Theorem 8.2]{Dosi08} $\tilde{\phi}$ extended to a local CP-map on $\mathcal{A}$. 
\end{proof}

\section{Irreducible representations and pure local CP-maps}
By a representation of a locally $C^*$-algebra $\mathcal{A}$ we always mean a local contractive $*$-homomorphism  from $\mathcal{A}$ into $\ced{E}{D}$ for some quantized domain $\mathcal{E}$. 
\begin{definition}
Let   $\pi : \mathcal{A}\rightarrow \mathcal{C}_{\mathcal{E}}^* (\mathcal{D})$ be a representation. The commutant of $\pi(\mathcal{A})$ is denoted by $\pi(\mathcal{A})'$ and is defined as $$\pi(\mathcal{A})'=\{T\in B(H): T\pi(a)\subseteq \pi(a)T, \text{  for all }a\in \mathcal{A}\}$$
\end{definition}

\begin{definition}\label{irreducible}
A representation $\pi : \mathcal{A}\rightarrow \mathcal{C}_{\mathcal{E}}^* (\mathcal{D})$ is said to be  irreducible if   $$\pi(\mathcal{A})'\cap \mathcal{C}_{\mathcal{E}}^* (\mathcal{D}) =\mathbb{C}I_{\mathcal{D}}$$
\end{definition}

The following  result is crucial  in our discussions. 
\begin{theorem}\label{irreducible imply minimal} Let $\mathcal{E}$ be a quantized Frechet domain. Let $\phi\in\cpcc{\mathcal{A}}{E}{D}$ and 
 $(\pi,V,\{H';\mathcal{E}';\mathcal{D}'\})$  be  a Stinespring representation of the map $\phi$. If $\pi$ is  irreducible, then $(\pi,V,\{H';\mathcal{E}';\mathcal{D}'\})$ is  a minimal Stinespring representation for the map $\phi$.
\end{theorem}
\begin{proof}
If possible assume that there exists an $l_{1}\in \mathbb{N}$ such that $[\pi(\mathcal{A})VH_{l_{1}}]\neq H_{l_{1}}'$. Since $V(H_{l})\subseteq H_{l}'$ and $H_{l}'$ is invariant for $\pi(a)$, for every $a\in \mathcal{A}$, we must have  $$[\pi(\mathcal{A})VH_{l_{1}}]\subsetneq H_{l_{1}}'.$$
Let $l_{0}=\min\{l\in \mathbb{N}: \pi(\mathcal{A})VH_{l}]\neq H_{l}' \}$.
Take $P$ to  be the orthognal projection of $H'$ onto the closed subspace  $[\pi(\mathcal{A})VH_{l_{0}}]$.
We claim that $P\in \pi(A)'\cap  \ced{E'}{D'}$. First, we prove that $P\in \ced{E'}{D'}$. 

To see $P(H_{l}')\subseteq H_{l}'$; let $l\in \mathbb{N}$. 
If $l\geq l_{0}$, then as  $\mathcal{E}'$ is an upward filtered family and $P$ is  a projection we must have $P(H_{l}')\subseteq [\pi(\mathcal{A})VH_{l_{0}}]\subsetneq H_{l_{0}}'\subseteq H_{l}'.$  
If $l< l_{0}$, then the choice of $l_{0}$ gives us  $H_{l}'=[\pi(\mathcal{A})VH_{l}]$. Then, 
to show $P(H_{l}')\subseteq H_{l}'$  it is enough to  show that $P(\pi(\mathcal{A})VH_{l}) \subseteq H_{l}'$. Since $l< l_{0}$ and $P$ is  a projection with range $[\pi(\mathcal{A})VH_{l_{0}}]$, we have
$H_{l}\subseteq H_{l_{0}}$.
Thus,
$$ \pi(\mathcal{A})VH_{l}\subseteq \pi(\mathcal{A})VH_{l_{0}} $$
$$P(\pi(\mathcal{A})VH_{l})=\pi(\mathcal{A})VH_{l}
\subseteq H_{l}' .$$
Hence $P(H_{l}')\subseteq H_{l}'$ for every $l\in \mathbb{N}$. 

Note that, as $P(H_{l}')\subseteq H_{l}'$ and $P$ is  a projection we have $P|_{H_{l}'}\in B(H_{l}')$.

Now, we show that $P(H_{l}'^{\perp}\cap \mathcal{D}')\subseteq H_{l}'^{\perp}\cap \mathcal{D}'$.  For $x\in H_{l}'^{\perp}\cap \mathcal{D}'$ and $y\in H_{l}'$ we need to show that $\langle Px, y \rangle =0$.
If $l<l_{0}$, then we have $H_{l}'=[\pi(\mathcal{A})VH_{l}]$.  Since $H_{l}'=[\pi(\mathcal{A})VH_{l}]\subseteq [\pi(\mathcal{A})VH_{l_{0}}] $, we have $Py=y$ for every $y\in H_{l}'$. Then it follows that 
$$ \langle Px, y \rangle =\langle x, Py \rangle=\langle x, y \rangle=0.$$
If $l\geq l_{0}$, then $[\pi(\mathcal{A})VH_{l_{0}}]\subsetneq H_{l}'$. Thus
$H_{l}'^{\perp}\cap \mathcal{D}' \subseteq [\pi(\mathcal{A})VH_{l_{0}}]^{\perp}$. It follows that  $Px=0$ for all $x\in H_{l}'^{\perp}\cap \mathcal{D}'$. Therefore $ \langle Px, y \rangle =0$ for all $y\in H_{l}'$. Hence $P\in \ced{E'}{D'}$.

To see  $P\in \pi(\mathcal{A})'$, let $a\in \mathcal{A}$. 
First, we  observe that $P\pi(a)h'=\pi(a)h'$ whenever $h'\in [\pi(\mathcal{A})VH_{l_{0}}]$.  As the restriction of $\pi(a)$ to $H_{l_{0}}'$ is a  bounded operator on $H_{l_{0}}'$, it is enough to consider  $h'$ in the dense subspace $\text{span}(\pi(\mathcal{A})VH_{l_{0}})$. Let $h'=\sum\limits_{i=1}^{n}\pi(a_{i})Vh_{i}$ for some $a_{i}\in \mathcal{A}$,  $h_{i}\in H_{l_{0}}$ and $n\in \mathbb{N}$,  $i=1,2,\cdots n$. Then,
\begin{align*}
     \pi(a)h' &= \pi(a)(\sum\limits_{i=1}^{n}\pi(a_{i})Vh_{i})\\
    &=\sum\limits_{i=1}^{n}\pi(aa_{i})Vh_{i}\in  [\pi(\mathcal{A})VH_{l_{0}}].
\end{align*}
It follows that $P\pi(a)h'=\pi(a)h'$ whenever $h'\in [\pi(\mathcal{A})VH_{l_{0}}]$.

Now, consider $h'\in \mathcal{D}'$. Write $h'=h'_{1}+h'_{2}$ where $h'_{1}\in [\pi(\mathcal{A})VH_{l_{0}}]$ and $h'_{2}\in [\pi(\mathcal{A})VH_{l_{0}}]^{\perp}\cap \mathcal{D}' $. It follows that  $P(h'_{2})=0$ and  $P\pi(a)h'_{1}=\pi(a)h'_{1}$.  Then
\begin{align*}
    \Vert P\pi(a)h'-\pi(a)Ph' \Vert ^2 &=  \Vert P\pi(a)(h'_{1}+h'_{2})-\pi(a)P(h'_{1}+h'_{2}) \Vert ^2 \\
    &=  \Vert P\pi(a)h'_{1}+P\pi(a)h'_{2}-\pi(a)Ph'_{1}+\pi(a)Ph'_{2} \Vert ^2 \\
    &=  \Vert P\pi(a)h'_{2} \Vert ^2 \\
    &=\langle P\pi(a)h'_{2}, P\pi(a)h'_{2} \rangle \\
    &= \langle \pi(a^*) P\pi(a)h'_{2}, h'_{2} \rangle.
\end{align*}
Let  $h'_{3} =P\pi(a)h'_{2} \in [\pi(\mathcal{A})VH_{l_{0}}]$. Then
$$\pi(a^*)h'_{3}= \pi(a^*)(\sum\limits_{i=1}^{n}\pi(a_{i})Vh_{i})
    =\sum\limits_{i=1}^{n}\pi(a^*a_{i})Vh_{i}\in  [\pi(\mathcal{A})VH_{l_{0}}].$$
But $h'_{2}\in [\pi(\mathcal{A})VH_{l_{0}}]^{\perp}$ will imply that $\langle \pi(a^*) h'_{3}, h'_{2} \rangle=0$. Hence $P\pi(a)h'=\pi(a)Ph'$ for every $a\in \mathcal{A}$ and $h'\in \mathcal{D}'$.
Hence  $P\in \pi(A)'\cap  \ced{E'}{D'}$. 

But $P\in \pi(A)'\cap  \ced{E'}{D'}$ is a contradiction as $\pi$ is irreducible, and  $P\neq 0$ and $P\neq I_{H}$. Hence $\pi$ is  a minimal Stinespring representation for $\phi$.
\end{proof}
\begin{remark}
It is well known that a representation $\theta$ of a $C^*$-algebra $\mathcal{C}$ is irreducible if and only if the commutant of $\theta(\mathcal{C})$ is trivial. If we take $\mathcal{A}$ to be a  $C^*$-algebra and  $\mathcal{E}=\{H\}$ in Definition \ref{irreducible}, then Definition \ref{irreducible} coincides with the usual definition of irreducible representations of  $C^*$-algebra. Also, our definition of irreducibility is motivated by the commutant considered  to establish a  Radon-Nikodym type theorem for local CP-maps in  \cite[Theorem 4.5]{Bhat21}.
\end{remark}

\subsection{Pure maps on local operator systems}
We introduce the notion of \textit{pure} local completely positive maps on local operator system  and study its connection with boundary representations for local operator systems.  For this, we use the convexity structure of the set $\cpcc{S}{E}{D}$.  
\begin{proposition} For a local operator system $S$, 
the set $\cpcc{S}{E}{D}$  is  a linear convex set. 
\end{proposition}
\begin{proof}
Let  $\phi_{1},\phi_{2}\in \cpcc{S}{E}{D}$ and $0 < t< 1$.  Fix $l\in \Omega$. There exist $\alpha_{r},\beta_{r}\in\Lambda$, $r=1,2$,  such that $$\phi_{r}^{(n)}([a_{ij}])\geq _{l} 0\text{ whenever } [a_{ij}]\geq_{\alpha_{r}} 0  \text{ and }  $$
$$\Vert \phi_{r}^{(n)}([a_{ij}])\Vert_{l} \leq p_{\beta_{r}}^{n}([a_{ij}]) \text{ for every } n\in \mathbb{N}.$$ 
Replace  $\phi_{1}$ and $\phi_{2}$ by $t\phi_{1}$ and $(1-t)\phi_{2}$ respectively. Then,  for $\alpha=\max\{\alpha_{1},\alpha_{2}\}$, we have  $$t\phi_{1}^{(n)}([a_{ij}])+(1-t)\phi_{2}^{(n)}([a_{ij}])\geq _{l} 0 \text{ whenever } [a_{ij}]\geq_{\alpha } 0.$$
Thus, $t\phi_{1}+(1-t)\phi_{2}$ is a local CP-map. To see its local CC, take $\beta=\max\{\beta_{1},\beta_{2}\}$. Then for every $[a_{ij}]\in M_{n}(S)$,
\begin{align*}
    \Vert t\phi_{1}^{(n)}([a_{ij}])+(1-t)\Vert \phi_{2}^{(n)}([a_{ij}])  \Vert_{l} &\leq \Vert t\phi_{1}^{(n)}([a_{ij}])\Vert_{l}+\Vert(1-t)\Vert \phi_{2}^{(n)}([a_{ij}])  \Vert_{l} \\
    &\leq t p_{\beta_{1}}^{n}([a_{ij}])+(1-t) p_{\beta_{2}}^{n}([a_{ij}]) \\
    &\leq p_{\beta}^{n}([a_{ij}]).
\end{align*} .
\end{proof}

\begin{definition}
A map $\phi\in \cpcc{S}{E}{D}$ is called pure if for any map $\psi \in  \cpcc{S}{E}{D}$ such that $\phi -\psi \in \cpcc{S}{E}{D}$, then there is a  scalar $t\in [0,1]$ such that $\psi=t\phi$.
\end{definition} 
\begin{remark}
A recent pre-print \cite{Joita21} also defines the notion of purity along similar lines.
\end{remark}

\begin{theorem}\label{pure iff irreducible}
A map $\phi \in \cpcc{\mathcal{A}}{E}{D}$ is pure if and only if $\phi$ is of the form $\phi(a)\subseteq V^*\pi(a)V$ for all $a\in\mathcal{A}$, where $\pi$ is an irreducible  representation of $\mathcal{A}$ on some quantized domain $\mathcal{E} '$ with its union space $\mathcal{D}'$ and $V\in L(\mathcal{D}, \mathcal{D}')$, $V\neq 0$ and $V(H_{l})\subseteq H_{l}'$ for all $l\in \Omega$. 
\end{theorem}
\begin{proof}
Let $\phi \in \cpcc{\mathcal{A}}{E}{D}$ be pure. Using  \cite[Theorem 5.1]{Dosi08}   we have a unital representation $\pi : \mathcal{A} \rightarrow \mathcal{C}_{\mathcal{E}'}^* (\mathcal{D}')$ for some  quantized domain $\mathcal{E}'$ with its union space $\mathcal{D}'$ such that $\phi(a)\subseteq V^* \pi(a)V$ where $V\in L(\mathcal{D},\mathcal{D}')$ and $V(H_{l})\subseteq H_{l}'$ for all $l\in \Omega$.  Clearly $V\neq 0$.  Now, let $T\in \pi(\mathcal{A})'\cap \mathcal{C}_{\mathcal{E}}^* (\mathcal{D})\text{ with }  0\leq T \leq I.$ Taking $\psi( .)= V^* T\pi( .)V|_{\mathcal{D}}$ in \cite[Theorem 4.5]{Bhat21} we have  $\psi\leq \phi$. As $\phi$ is pure it follows that  $\psi=t\phi$. Applying \cite[Corollary 4.6]{Bhat21}, $T=tI$. Hence $\pi$ is irreducible.

Conversely, let $\pi$ be an irreducible representation of $\mathcal{A}$ on some quantized domain $\mathcal{E}'$ with its union space $\mathcal{D}$  and $V$ be a non zero operator in $L(\mathcal{D},\mathcal{D}')$ such that $V(H_{l})\subseteq H_{l}'$ for all $l\in \Omega$. To show that $\phi(.)\subseteq V^*\pi(.)V$ is pure, consider $\psi\in \cpcc{\mathcal{A}}{E}{D}$ with $\psi\leq \phi$. As $\pi$ is  irreducible by Theorem \ref{irreducible imply minimal}  $(\pi, V,\{H',\mathcal{E}',\mathcal{D}'\})$ is  a minimal Stinespring representation's representation  for  $\phi.$ Now, applying \cite[Corollary 4.6]{Bhat21}, there exists a unique $T\in \pi(\mathcal{A})'\cap \mathcal{C}_{\mathcal{E}}^* (\mathcal{D})$ such that $0\leq T \leq I$ and   $\psi(a)\subseteq  V^*T\pi(a)V$ for all $a\in \mathcal{A}$. 
Since $\pi$ is irreducible, $T=tI$. It follows that $\psi= t\phi$  and hence $\phi$ is pure.
\end{proof}

\begin{proposition}\label{linear extreme and pure}
Let $S_{1}$ and $S_{2}$ be  local operator systems in a locally $C^*$-algebra $\mathcal{A}$ such that $S_{1}\subseteq S_{2}$. Let $\phi :S_{2}\rightarrow \ced{E}{D}$ be  a unital local CP-map such that its a linear extreme point of $\cpcc{S_{2}}{E}{D}$. If $\phi|_{S_{1}}$ is pure, then $\phi$ is a pure.
\end{proposition}
\begin{proof}
Let $\phi_{1},\phi_{2}\in\cpcc{S_{2}}{E}{D}$ such that $\phi=\phi_{1}+\phi_{2}$.  Since $\phi |_{S_{1}}$ is pure, there exists $t\in(0,1)$ such that $\phi_{1}|_{S_{1}}=t\phi|_{S_{1}}$ and $\phi_{2}|_{S_{1}}=(1-t)\phi|_{S_{1}}$. The maps $\frac{1}{t}\phi_{1}$ and $\frac{1}{1-t}\phi_{2}$ are unital local CP-map on $S_{2}$. By Theorem \ref{local CC iff local CP} both the maps are local CC-maps. It follows that $\frac{1}{t}\phi_{1} ,\frac{1}{1-t}\phi_{2}\in\cpcc{S_{2}}{E}{D}$. Then the expression $\phi=t \frac{1}{t}\phi_{1}+ (1-t)\frac{1}{1-t}\phi_{2}$ and the assumption $\phi$ is linear extreme implies that $\phi$ is pure.  
\end{proof}

\section{Local Boundary representations}
In this section, we introduce the notion of local boundary representations for locally $C^*$-algebras and establish its connection with pure local CP-maps.  
\begin{definition}\label{local unique extension property}
Let $S$ be a linear subspace of a locally $C^*$-algebra $\mathcal{A}$ such that $S$ generates $\mathcal{A}$. A representation $\pi : \mathcal{A}\rightarrow \mathcal{C}_{\mathcal{E}}^* (\mathcal{D})$   is said to have \textit{local unique extension property} for $S$
if $\pi|_{S}$ has a  unique local completely positive extension to $\mathcal{A}$, namely $\pi$ itself.  
\end{definition}
\begin{remark}
Let $\pi: \mathcal{A}\rightarrow\ced{E}{D}$ be  a representation  of $\mathcal{A}$. Then  $\pi|_{S}$ has just one multiplicative local  CP-extension to $\mathcal{A}$, namely $\pi$ itself, but in general, there may exist other local CP-extensions of $\pi|_{S}$.
\end{remark}

\begin{example}
For a self adjoint operator $T\in \ced{E}{D}$, let $S=\text{span}\{I,T, T^2\}$ and $\mathcal{B}$ be the locally $C^*$-algebra generated by $S$ in $\ced{E}{D}$. We show that the identity representation $I_{\mathcal{B}}$ of $\mathcal{B}$ has local unique extension property. Let $\phi:\mathcal{B}\rightarrow \ced{E}{D}$ be a local completely positive map such that $\phi(x)=x$ for all $x\in S$. Consider a minimal Stinespring  representation $(\pi,V,\{H';\mathcal{E}';\mathcal{D}'\})$ of $\phi$. To prove $\phi=I_{\mathcal{B}}$ on $\mathcal{B}$ it is enough to show that $V$ is a unitary. We claim  that $V(\mathcal{D})$ is invariant for  $\pi(\mathcal{B})$. Then by minimality $H'=[\pi(\mathcal{B})V(\mathcal{D})]\subseteq [V(\mathcal{D})] \subseteq H'$ will imply $V$ is a unitary. Now, to see the claim let us first show that $\pi(T)V(\mathcal{D})\subseteq V(\mathcal{D}) $. For that, we show that $\pi(T)V(H_{l})\subseteq V(H_{l})$ for every $l\in \Omega$.  Let $l\in \Omega$ and $g\in H'_{l}$,
\begin{align*}
   \Vert (I-VV^*)\pi(T)VV^*g \Vert ^2 &=\langle (I-VV^*)\pi(T)VV^*g,~ (I-VV^*)\pi(T)VV^*g\rangle \\
   &= \langle VV^*\pi(T)(I-VV^*)\pi(T)VV^*g,~ g   \rangle \\
   &= \langle VV^*\pi(T)\pi(T)VV^*g-VV^*\pi(T)VV^*\pi(T)VV^*g,~ g   \rangle \\
   &= \langle V\phi(T^2)V^*h'-V\phi(T)\phi(T)V^*g,~ g   \rangle \\
  & = \langle VT^2V^*g-VT^2V^*h',~ g   \rangle \\
  & = 0.
\end{align*}
Thus $(I-VV^*)\pi(T)VV^* =0$ on $H_{l}'$. Since $T$ is self adjoint,  $VV^*\pi(T)(I-VV^*) =0$ on $H_{l}'$. These two observations and the facts  $\pi(T)_{|_{H_{l}'}}\in B(H_{l}')$, $V_{|_{H_{l}}}$ is an isometry  and $\pi(T)V(H_{l})\subseteq H_{l}'$  will give $\pi(T)V(H_{l})\subseteq V(H_{l})$. As $l$ is arbitrary, it follows that $\pi(T)V(\mathcal{D})\subseteq V(\mathcal{D})$. 
To show $\pi(\mathcal{B})V(\mathcal{D})\subseteq V(\mathcal{D})$, let $T_{0}\in \mathcal{B}$ and $\mathcal{B}_{T}=span\{I,T,T^2,T^3,\cdots\}$. Then $T_{0}=\lim T_{\lambda}$, where $T_{\lambda}\in \mathcal{B}_{T}$. 
For $h\in H_{l}$,
\begin{align*}
    \Vert \pi(T_{\lambda})Vh - \pi(T_{0})Vh \Vert_{H'_{l}} &=  \Vert \pi(T_{\lambda} -T_{0})Vh \Vert_{H'_{l}} \\
    &\leq p_{\alpha}(T_{\lambda} -T_{0})\Vert h\Vert_{H_{l}},
\end{align*}
where $\alpha$ corresponds to $l$ in the local contractivity of $\pi$. As $\{T_{\lambda}\}$ converges to $T_{0}$, we have $p_{\alpha}(T_{\lambda} -T_{0})\rightarrow 0$ and hence $\{\pi(T_{\lambda})Vh\}$ converges to $\pi(T_{0})Vh$ in $H'_{l}$. Therefore, $$ \pi(T_{0})Vh \in [\pi(T_{\lambda})Vh] .$$
As $\pi(T)$ leaves $V(H_{l})$ invariant, so is every element of $\mathcal{B}_{T}$. Then using the fact that $VH_{l}$ is  a closed subspace (as $V$ is an isometry and $H_{l}$ is a closed subspace),
$$ \pi(T_{0})Vh \in [\pi(T_{\lambda})Vh] \subseteq [\pi(T_{\lambda})V(H_{l})]\subseteq  [V(H_{l})]=  V(H_{l}).$$
Therefore $\pi(\mathcal{B})V(H_{l})\subseteq  V(H_{l})$ for every $l$ and hence $\pi(\mathcal{B})V(\mathcal{D})\subseteq  V(\mathcal{D})$.
\end{example}

\begin{example}
Let $K$ be an infinite dimensional separable complex Hilbert space with  a complete orthonormal basis $\{e_{n}:n\in \mathbb{N}\}$. Consider $K_{n}=span\{e_{1},e_{2},\cdots e_{n}\}$  and $H_{n}=K\oplus K_{n}$. Then  $\mathcal{E}=\{H_{n}:n\in \mathbb{N}\}$  is  a quantized domain in the Hilbert space $H=K\oplus K$ with union space $\mathcal{D}=\cup\{H_{n}:n\in \mathbb{N}\}$. Define $V:H\rightarrow H$  to be the map $V_{0}\oplus 1_{K}$ where $V_{0}:K\rightarrow K$ be the unilateral right shift operator and $1_{K}$ be the identity operator on $K$. 
Note that $V$ is an isometry  but not a unitary. Also, $V(K\oplus K_{n})\subseteq K\oplus K_{n}$ and  $$V((K\oplus K_{n})^{\perp}) =V(0\oplus K_{n}^{\perp})=0\oplus K_{n}^{\perp}= (K\oplus K_{n})^{\perp}.$$ Therefore $V|_{\mathcal{D}}\in\ced{E}{D}$.

Consider the local operator system $S=span\{1_{\mathcal{D}}, V|_{\mathcal{D}}, V^*\}$ in $\ced{E}{D}$ and let $\mathcal{B}$  the locally $C^*$-algebra generated by $S$ in $\ced{E}{D}$. We claim that the inclusion map from $S$ to $\ced{E}{D}$ have two distinct local CP-extension to $\mathcal{B}$. Obviously the  inclusion  representation $I_{\mathcal{B}}:\mathcal{B}\rightarrow \ced{E}{D}$ is  a local CP-extension of the inclusion map on $S$.  Define $\psi:\mathcal{B}\rightarrow \ced{E}{D}$ by $\psi(a)=V^*I_{\mathcal{B}}(a)V|_{\mathcal{D}}$ for all $a\in \mathcal{B}$. Clearly $\psi$ is  a unital local completely positive map on $\mathcal{B}$.
For all scalars $c_{1},c_{2}$ and $c_{3}$ we have
\begin{align*}
\psi(c_{1}1_{\mathcal{D}}+c_{2} V|_{\mathcal{D}}+c_{3} V^*)&=V^*(c_{1}1_{\mathcal{D}}+c_{2} V|_{\mathcal{D}}+c_{3} V^*) V|_{\mathcal{D}}\\
&= c_{1}1_{\mathcal{D}}+c_{2} V|_{\mathcal{D}}+c_{3} V^*.
\end{align*}
Therefore $\psi|_{S}=I_{\mathcal{B}}|_{S}$. Now the element $V|_{\mathcal{D}}V^*\in \mathcal{B}$. But  $$\psi(V|_{\mathcal{D}}V^*)=V^*V|_{\mathcal{D}}V^*V|_{\mathcal{D}}=I_{\mathcal{D}}\neq V|_{\mathcal{D}}V^*.$$
That is $\psi \neq I_{\mathcal{B}}$ on $\mathcal{B}$. Therefore,  the irreducible representation $I_{\mathcal{B}}$ doesn't have local unique extension property for $S$. 
\end{example}

\begin{definition}\label{local boundary representation}
Let $S$ be a linear subspace of a local $C^*$-algebra $\mathcal{A}$ such that $S$ generates $\mathcal{A}$. An  irreducible representation $\pi : \mathcal{A}\rightarrow \mathcal{C}_{\mathcal{E}}^* (\mathcal{D})$   is called a \textit{local boundary representation for} $\mathcal{S}$ if $\pi$ has local unique extension property for $S$. 
\end{definition}

\begin{remark}
The Definition \ref{local unique extension property} and Definition \ref{local boundary representation} are meaningful for local operator systems in arbitrary locally $C^*$-algebras. But the Arveson's extension theorem in the context of locally $C^*$-algebras is available only for $\ced{E}{D}$  for quantized Frechet domain $\mathcal{E}$ and thus we restrict our studies to the context of  Frechet locally C$^*$-algebras.
\end{remark}

Now, we show that the local boundary representations are intrinsic invariants for local operator systems. Let  $\mathcal{A}_{1}$ be a  locally $C^*$-algebra and  $\mathcal{A}_{2}=C^*_{\mathcal{E}_{2}}(\mathcal{D}_{2})$  be the locally $C^*$-algebras of all non-commutative  continuous functions on a quantized Frechet domain  $\mathcal{E}_{2}$ with its union space $\mathcal{D}_{2}$.
\begin{theorem}
Let $S_{1}$ and $S_{2}$ be linear subspaces of $\mathcal{A}_{1}$ and $\mathcal{A}_{2}$ respectively. Let $\phi :S_{1}\rightarrow S_{2}$ be a unital surjective local completely isometric linear map. Then for every boundary representation $\pi_{1}$ of $\mathcal{A}_{1}$ there exists a boundary representation $\pi_{2}$ of $\mathcal{A}_{2}$ such that $\pi_{2}\circ\phi(a)=\pi_{1}(a)$ $\forall ~ a\in S_{1}$.  
\end{theorem} 
\begin{proof}
By Theorem  \ref{local CP-extension of local CC-map} we can extend $\phi$ to a local CP-map $\tilde{\phi}:\mathcal{A}_{1}\rightarrow \mathcal{A}_{2}$. Consider the map $\psi : S_{2}\rightarrow \ced{E}{D}$ given by $(\psi\circ\phi)(a)=\pi_{1}(a)$. Clearly $\psi$ is  a unital local CC-map. Again by  Theorem \ref{local CP-extension of local CC-map} there exists a local CP-extension of $\psi$, say $\pi_{2}$, where $\pi_{2}:\mathcal{A}_{2}\rightarrow \ced{E}{D}$ such that $(\pi_{2}\circ\phi)(a)=\pi_{1}(a)$ for every $a\in S_{1}$. Since $\pi_{1}$ is  a boundary representation, $(\pi_{2}\circ\phi)(a)=\pi_{1}(a)$ for every $a\in \mathcal{A}_{1}$.  Note that the locally $C^*$-algebra generated by $\tilde{\phi}(\mathcal{A}_{1})$ is equal to $\mathcal{A}_{2}$ and $\pi_{2}$ is continuous for the respective topologies. Thus, to prove $\pi_{2}$ is an algebra homomorphism it's enough to prove that $\pi_{2}(xy)=\pi_{2}(x)\pi_{2}(y)$ for every $x\in \tilde{\phi}(\mathcal{A}_{1})$ and for all $y\in \mathcal{A}_{2}$. But in view of \cite[Corollary 5.5]{Dosi08}, it's enough to prove that 
$$\pi_{2}(x)^*\pi_{2}(x)=\pi_{2}(x^*x)~~\forall x\in \tilde{\phi}(\mathcal{A}_{1}).$$
Let $a\in\mathcal{A}_{1}$. Then using the fact that a local positive  map is positive \cite[Proposition 2.1]{Joita21positivity}, and   \cite[Corollary 5.5]{Dosi08} we have, on $\mathcal{D}$,
\begin{align*}
\pi_{2}(\tilde{\phi}(a))^*\pi_{2}(\tilde{\phi}(a)) &\leq \pi_{2}(\tilde{\phi}(a)^*\tilde{\phi}(a)) = \pi_{2}(\tilde{\phi}(a^*)\tilde{\phi}(a)) \\
 &\leq \pi_{2}(\tilde{\phi}(a^*a)) \\
 &= \pi_{1}(a^*a) \\
 &= \pi_{1}(a^*) \pi_{1}(a) \\
 &= \pi_{2}(\tilde{\phi}(a))^* \pi_{2}(\tilde{\phi}(a)).
\end{align*}
Therefore  $\pi_{2}(\tilde{\phi}(a)^*\tilde{\phi}(a))=\pi_{2}(\tilde{\phi}(a))^* \pi_{2}(\tilde{\phi}(a))$  on $\mathcal{D}$. Thus $\pi_{2}$ is  a representation of $\mathcal{A}_{2}$. In fact we proved that any local CP-extension of $\psi=\pi_{2}{|_{S_{2}}}$ to $\mathcal{A}_{2}$ is  multiplicative on $\mathcal{A}_{2}$. Equivalently, $\pi_{2}$ has local unique extension property for $S_{2}$. 

Now, note that $\pi_{1}(\mathcal{A}_{1})\subseteq (\pi_{2}\circ\tilde{\phi})(\mathcal{A}_{1}) \subseteq \pi_{2}(\mathcal{A}_{2})$. Thus, for commutants we have  $\pi_{2}(\mathcal{A}_{2})' \subseteq \pi_{1}(\mathcal{A}_{1})'$. Then the irreducibility of  $\pi_{2}$ follows from the irreducibility of $\pi_{1}$. This completes the proof.
\end{proof}
\begin{corollary}\label{invariance of boundary representation}
Let $S_{1}$ and $S_{2}$ be local operator systems of $\mathcal{A}_{1}$ and $\mathcal{A}_{2}$ respectively. Let $\phi :S_{1}\rightarrow S_{2}$ be a unital invertible local CP-map  such that $\phi^{-1}$ is also a local CP-map. Then for every boundary representation $\pi_{1}$ of $\mathcal{A}_{1}$ there exists a boundary representation $\pi_{2}$ of $\mathcal{A}_{2}$ such that $\pi_{2}\circ\phi(a)=\pi_{1}(a)$ $\forall ~ a\in S_{1}$. 
\end{corollary}

\begin{remark}
We expect the above theorem and consequently the corollary to be true for any Frechet locally $C^*$-algebras  in place of $\mathcal{A}_{2}=C^*_{\mathcal{E}_{2}}(\mathcal{D}_{2})$. 
\end{remark}

\subsection{Characterisation of boundary representations}
The following theorem shows that the restriction of  a local boundary representation to the local operator system is a pure map. 

\begin{theorem}\label{restriction pure}
Let $S$ be a local operator system in a Frechet local $C^*$-algebra $\mathcal{A}$ such that $S$ generates $\mathcal{A}$.  Let $\mathcal{E}$ be a quantized Frechet domain with its union space $\mathcal{D}$, and $\pi :\mathcal{A}\rightarrow \ced{E}{D}$  be a boundary representation for $S$. Then $\pi |_{S}$ is a pure map on $S$.
\end{theorem}
\begin{proof}
Let  $\pi_{1},\pi_{2}\in\cpcc{S}{E}{D}$ such that $\pi|_{S}=\pi_{1}+\pi_{2}$. Then by Dosiev-Arveson extension theorem \cite[Theorem 8.2]{Dosi08}, each $\pi_{i}$ extends to  a local CPCC map on $\mathcal{A}$, call it $\tilde{\pi}_{i}$, $i=1,2$. We will show that  $\tilde{\pi}_{1}+\tilde{\pi}_{2}\in \cpcc{\mathcal{A}}{E}{D}$. For that, fix $l\in \mathbb{N}$. Then there exists $\alpha_{i}$ and $\beta_{i}$ such that $\tilde{\pi}_{i}(a)\geq_{l}0$ whenever $a\geq_{\alpha_{i}}0$ in $\mathcal{A}$ and $\Vert \tilde{\pi}_{i}(b) \Vert_{l}\leq p_{\beta_{i}}(b)$ for every $b\in \mathcal{A}$. Take $\alpha =\max\{\alpha_{1},\alpha_{2}\}$ and $\beta =\max\{\beta_{1},\beta_{2}\}$. Using the fact that the family of semi-norms $\{p_{n}\}_{n\in\mathbb{N}}$ is an  upward filtered family, we have $\tilde{\pi}_{i}(a)\geq_{l}0$ whenever $a\geq_{\alpha} 0$ in $\mathcal{A}$ and $\Vert \tilde{\pi}_{i}(b) \Vert_{l}\leq p_{\beta}(b)$ for every $b\in \mathcal{A}$. Therefore, $\tilde{\pi}_{1}+\tilde{\pi}_{2}\in \cpcc{\mathcal{A}}{E}{D}$. 

Now, since $\tilde{\pi}_{1}+\tilde{\pi}_{2}|_{S}=\pi_{1}+\pi_{2}=\pi|_{S}$ and $\pi$ is  a boundary representation for $S$, we must have $\pi(a)=\tilde{\pi}_{1}(a)+\tilde{\pi}_{2}(a)$ for every $a\in \mathcal{A}$. The irreducibility of $\pi$ and  the  Theorem \ref{pure iff irreducible} implies that $\pi$ is a pure map. Thus, for each $i$,  there exist $t_{i}\in [0,1]$ such that $\tilde{\pi}_{i}(a)=t_{i}\pi(a)$ for every $a\in \mathcal{A}$. It follows that    $\pi_{i}=t_{i}\pi|_{S}$. Hence $\pi|_{S}$ is a pure map on $S$.
\end{proof}

Now, we show that certain irreducible representations  of $\mathcal{A}$ that are pure CPCC-maps on $S$ are local boundary representations. For this, we need to introduce a couple of new notions. Let $S$ be a local operator system in  a  local $C^*$-algebra $\mathcal{A}$ such that $\mathcal{A}$ is generated by $S$, and let $\pi:\mathcal{A}\rightarrow \ced{E}{D}$ be a  representation of $\mathcal{A}$. We say that $\pi$ is  a \textit{finite representation for } $S$ if  for every isometry $V\in B(H)$ with $V(H_{l})\subseteq H_{l} $ for every $l\in \Lambda$, the condition $\pi(x)\subseteq V^*\pi(x)V$ for every $x\in S$ implies $V$ is  a unitary. We say that the  local operator system $S$  \textit{separates} the irreducible representation $\pi$ if for any  irreducible representation  $\rho$ of $\mathcal{A}$ on some quantized domain $\mathcal{E}'$ with its union space $\mathcal{D}'=\bigcup\limits_{l\in \Lambda} H_{l}'$ and an isometry $V$ in $B(H,H')$ that satisfies $V(\mathcal{H}_{l})\subseteq \mathcal{H}_{l}'$  for every $l\in \Lambda$ such that  $\pi(x)\subseteq V^*\rho(x)V$ for all $x\in S$ implies that $\pi$ and $\rho$ are unitarily equivalent representations of $\mathcal{A}$.

\begin{theorem}
Let $S$ be a local operator system in  a  local $C^*$-algebra $\mathcal{A}$ such that $\mathcal{A}$ is generated by $S$. Then, an irreducible representation $\pi:\mathcal{A}\rightarrow \ced{E}{D}$   is a local boundary representations for $S$  if and only if the following conditions hold;
\begin{enumerate}[(i)] 
    \item $\pi|_{S}$ is a pure map on $S$
    \item Every local CP-extension of $\pi|_{S}$ to $\mathcal{A}$ is a linear extreme 
    point of \\ $\cpcc{\mathcal{A}}{E}{D}$
    \item $\pi$ is  a finite representation for $S$
    \item $S$ separates $\pi$.
\end{enumerate}
\end{theorem}
\begin{proof}
Let $\pi$ be an irreducible representation of $\mathcal{A}$. Assume that $\pi$ is a local boundary representation for $S$. Then the statement $(i)$ follows by  Theorem \ref{restriction pure}. 

 $(ii)$: Since $\pi$ is  a local boundary representation, there is only one local CP-extension  of $\pi|_{S}$ to $\mathcal{A}$, namely $\pi$ itself. Let $\phi_{1},\phi_{2}\in \cpcc{\mathcal{A}}{E}{D}$ such that $\pi=\phi_{1}+\phi_{2}$. Then 
$\pi |_{S}=\phi_{1}|_{S}+\phi_{2}|_{S}$.  But $\pi|_{S}$ is pure by statement $(i)$. Thus $\phi_{1}|_{S}=t\pi|_{S}$ and $\phi_{2}|_{S}=(1-t)\pi|_{S}$ for some $t\in [0,1]$. If $0<t<1$, then $\pi|_{S}=\frac{1}{t}\phi_{1}|_{S}$ and $\pi|_{S}=\frac{1}{1-t}\phi_{2}|_{S}$. Now the maps $\frac{1}{t}\phi_{1}$ and $\frac{1}{1-t}\phi_{2}$ on $\mathcal{A}$ are unital local CP-extensions of $\pi|_{S}$. But $\pi$ is a boundary representation for $S$ would imply that   $\pi=\frac{1}{t}\phi_{1}$ and $\pi=\frac{1}{1-t}\phi_{2}$ on $\mathcal{A}$.
That is, $\pi$ is  a linear extreme point of $\cpcc{\mathcal{A}}{E}{D}$.

$(iii)$: Consider an isometry $V$ on $H$ such that $\pi(x)\subseteq V^*\pi(x)V$ for every $x\in S$ and $V(H_{l})\subseteq H_{l} $ for every $l\in \Lambda$. Then $\phi(a):=V^*\pi(a)V|_{\mathcal{D}}$ for all $a\in \mathcal{A}$ is a unital local CP-extension of $\pi|_{S}$. As $\pi$ is a local boundary representation we must have $\pi(a)= V^*\pi(a)V|_{\mathcal{D}}$ for all $a\in \mathcal{A}$. We claim that $V\in\pi(\mathcal{A})'\cap\ced{E}{D}$.  Clearly $V$ is bounded and $V(H_{l})\subseteq H_{l}$ $\forall$ $l$.  Let $x\in H_{l}^{\perp}\cap\mathcal{D}$. Since $\pi$ is irreducible, by Theorem \ref{irreducible imply minimal} $(\pi,V,\{H,\mathcal{E},\mathcal{D}\})$ is a minimal Stinespring for $\pi$. Then by \cite[Lemma 4.2]{Bhat21}, $Vx=\pi(1)Vx\in H_{l}^{\perp}$. It follows that $Vx\in H_{l}^{\perp}\cap\mathcal{D}$  as $V(H_{l})\subseteq H_{l}$. Thus $V(H_{l}^{\perp}\cap\mathcal{D})\subseteq H_{l}^{\perp}\cap\mathcal{D}$ and hence  $V\in \ced{E}{D}$. To see $V\in \pi(\mathcal{A})'$; first note that $dom(V\pi(a))=\mathcal{D}\subseteq dom(\pi(a)V)$ for all $a\in \mathcal{A}$. Let $h\in \mathcal{D}$ and $a\in \mathcal{A}$.  

$\Vert V\pi(a)h -\pi(a)Vh \Vert^2 $
\begin{equation*}
 \begin{split}  
 =& \langle  V\pi(a)h -\pi(a)Vh ,  V\pi(a)h -\pi(a)Vh \rangle \\
    =& \Vert V\pi(a)h \Vert^2 - \langle \pi(a)Vh ,  V\pi(a)h \rangle - \langle V\pi(a)h ,  \pi(a)Vh  \rangle + \Vert \pi(a)Vh \Vert ^2 \\
    =& \Vert \pi(a)h \Vert^2 - \langle V^*\pi(a)Vh ,\pi(a)h \rangle - \langle \pi(a)h, V^*\pi(a)Vh \rangle + \Vert  \pi(a)Vh\Vert^2 \\
    =& \Vert \pi(a)h \Vert ^2 - \langle \pi(a)h ,  \pi(a)h \rangle  - \langle \pi(a)h ,  \pi(a)h \rangle + \Vert \pi(a)Vh \Vert ^2 \\
    =&  \Vert \pi(a)Vh \Vert ^2  -\Vert \pi(a)h  \Vert ^2 =\langle \pi(a)Vh , \pi(a)Vh \rangle - \langle \pi(a)h , \pi(a)h \rangle \\
    =& \langle  V^*\pi(a^*)\pi(a)Vh , h \rangle - \langle \pi(a)^*\pi(a)h , h \rangle \\
    =& \langle  \pi(a^*a)h , h \rangle - \langle \pi(a^*a)h , h \rangle =0.
\end{split}
\end{equation*}
Therefore $V\pi(a)\subseteq \pi(a)V$ for every $a\in \mathcal{A}$ and hence $V\in \pi(\mathcal{A})'\cap\ced{E}{D}$. By the irreducibility of $\pi$ implies  $V=\lambda I_{H}$, $\lambda\in \mathbb{C}$. Thus, the isometry $V$ is a unitary. Hence $\pi$ is  a finite representation for $S$.

 $(iv)$: Assume that $\rho$ is an irreducible representation of $\mathcal{A}$  on some quantized domain $\mathcal{E}'$ with its union space $\mathcal{D}'=\bigcup\limits_{l\in \Lambda} H_{l}'$ and an isometry $V$ in $B(H,H')$ that satisfies $V(\mathcal{H}_{l})\subseteq \mathcal{H}_{l}'$  for every $l\in \Lambda$ such that  $\pi(x)\subseteq V^*\rho(x)V$ for all $x\in S$. As $\pi$ is  a local boundary representation for $S$, it follows that $\pi(a)\subseteq V^*\rho(a)V$ for all $a\in \mathcal{A}$ . 
Here $\pi$ and $\rho$ are irreducible representations of $\mathcal{A}$. By Theorem \ref{irreducible imply minimal} the Stinespring representations $(\pi, I_{H}, \{H,\mathcal{E},\mathcal{D}\})$ and $(\rho, V, \{H',\mathcal{E}',\mathcal{D}'\})$ are minimal  for $\pi$. Then  \cite[Theorem 3.4]{Bhat21} will imply that $\pi$ and $\rho$ are unitarily equivalent.
Hence  $S$ separate $\pi$.

Conversely assume that the irreducible representation $\pi$ satisfies all the four conditions.  Let $\phi:\mathcal{A}\rightarrow \ced{E}{D}$  be a local CP-map such that $\phi(a)=\pi(a)$  for every  $a\in S$. By condition $(ii)$, $\phi$ is  a linear extreme point of $\cpcc{\mathcal{A}}{E}{D}$. Then statement $(i)$  and Proposition \ref{linear extreme and pure} will imply that $\phi$ is a pure map in $\cpcc{\mathcal{A}}{E}{D}$. If $\{\omega; V;\{K,\mathcal{F},\mathcal{O}\}\}$ is a minimal Stinespring representation for $\phi$,   then by Theorem  \ref{pure iff irreducible} $\omega$ is irreducible. Also,   $$\pi(a)=\phi(a)=V^*\omega(a) V|_{\mathcal{D}} \text{ for all }a\in S.$$ 
As $\pi$ separates $S$,  $\pi$ and $\omega$ are unitarily equivalent. Let $U:K\rightarrow H$ be a  unitary  such that  $U(\mathcal{O})\subseteq \mathcal{D}$ and $$\omega(a)= U^*\pi(a) U|_{\mathcal{D}} \text{ for all }a\in S.$$
Then 
$$\pi(a)= V^*U^*\pi(a) UV|_{\mathcal{D}} \text{ for all }a\in S.$$
Since $\pi$ is a finite representation and $UV$ is an isometry on $H$, we have $UV$ is a unitary. Thus $V=U^*(UV)$ is  also a unitary. Therefore $\phi(a)=V^*\pi(a)V|_{D}$ on $\mathcal{A}$ is a representation of $\mathcal{A}$ which coincides with $\pi$ on $S$. Therefore $\phi(a)=\pi(a)$ for all $a\in \mathcal{A}$ and hence $\pi$ is  a local boundary representation for $S$.
\end{proof}

{\bf Acknowledgments.} 
The author would like to thank his supervisor Prof. A.K. Vijayarajan  for some useful discussions. Also, the author is thankful to the  Department of Atomic Energy, Government of India for the NBHM Ph.D. fellowship (File No. 0203/17 /2019/R\&D-II/10974).

\bibliographystyle{amsplain}

\end{document}